\documentclass[11pt]{article}

\usepackage{graphicx}
\usepackage{amsthm,amsmath,amssymb}
\usepackage{url}
\usepackage{tikz-3dplot-circleofsphere}
\usepackage{fullpage}

\newtheorem{theorem}{Theorem}
\newtheorem{lemma}[theorem]{Lemma}
\newtheorem{corollary}[theorem]{Corollary}
\newtheorem{conjecture}[theorem]{Conjecture}

\newcommand{\RR}{\mathbb{R}}
\renewcommand{\SS}{\mathbb{S}}

\DeclareMathOperator{\crn}{cr}
\newcommand{\crg}{\crn_{\SS^2}}

\begin{document}

\title{On a conjecture by Anthony Hill}

\author{
  Bojan Mohar\thanks{Supported in part by the NSERC Discovery Grant R611450 (Canada),
  by the Canada Research Chairs program,
  and by the Research Project J1-8130 of ARRS (Slovenia).}\\[1mm]
  Department of Mathematics\\
  Simon Fraser University\\
  Burnaby, BC, Canada\\
  {\tt mohar@sfu.ca}
}

\maketitle

\begin{abstract}
In the 1950's, English painter Anthony Hill described drawings of complete graphs $K_n$ in the plane having precisely
$$H(n) = \tfrac{1}{4}\lfloor \tfrac{n}{2}\rfloor \, \lfloor \tfrac{n-1}{2}\rfloor \, \lfloor \tfrac{n-2}{2}\rfloor \,\lfloor \tfrac{n-3}{2}\rfloor$$
crossings. It became a conjecture that this number is minimum possible and, despite serious efforts, the conjecture is still widely open. Another way of drawing $K_n$ with the same number of crossings was found by Bla\v zek and Koman in 1963. In this note we provide, for the first time, a very general construction of drawings attaining the same bound. Surprisingly, the proof is extremely short and may as well qualify as a ``book proof". In particular, it gives a very simple explanation of the phenomenon discovered by Moon in 1968 that a random set of $n$ points on the unit sphere $\SS^2$ in $\RR^3$ joined by geodesics gives rise to a drawing whose number of crossings asymptotically approaches the Hill value $H(n)$.
\end{abstract}

\section{Introduction}

English painter Anthony Hill made an extraordinary conjecture in the 1950's that remained unanswered until today despite numerous attacks using powerful machinery in trying to resolve his conjecture. Starting with a question how to draw $\binom{n}{2}$ connections between $n$ objects so that the painting would involve a minimum number of under or over-crossings lead to the notion of the \emph{crossing number} of a graph \cite{HaHi62}. Hill suggested a drawing that involved
\begin{equation}\label{eq:Hill}
   H(n) = \tfrac{1}{4}\lfloor \tfrac{n}{2}\rfloor \, \lfloor \tfrac{n-1}{2}\rfloor \, \lfloor \tfrac{n-2}{2}\rfloor \,\lfloor \tfrac{n-3}{2}\rfloor =
   \left\{
       \begin{array}{ll}
         \tfrac{1}{64} n(n-2)^2(n-4), & \hbox{$n$ is even;} \\[1mm]
         \tfrac{1}{64} (n-1)^2(n-3)^2, & \hbox{$n$ is odd}
       \end{array}
   \right.
\end{equation}
crossings. Hill's drawings of complete graphs are called \emph{cylindrical drawings} because they can be realized on a cylinder in such a way that all vertices lie on the two circles forming the cylinder and no edge crosses those two circles.
Soon after these drawings were published in \cite{HaHi62}, Bla\v zek and Koman \cite{BlKo63} found another kind of drawings of complete graphs involving precisely the same number of crossings. Their drawings correspond to \emph{2-page drawings} in which the vertices are drawn on the boundary of a unit disk in the plane and no edge crosses this boundary, so each edge is entirely inside the disk or entirely outside. It has been proved quite recently that no cylindrical \cite{AAFRS14} and no 2-page drawing \cite{AAFRS13} of $K_n$ has fewer than $H(n)$ crossings, thus giving the first real support to the conjecture of Hill.

We will say that a drawing $D$ of the complete graph $K_n$ is a \emph{Hill drawing} if it has precisely $H(n)$ crossings.

No other Hill drawings of complete graphs have been discovered until 2014 when \'Abrego et al.~\cite{AAFRV14} described modifications of cylindrical drawings of $K_{2n}$ yielding Hill drawings of $K_{2n+1}$, $K_{2n+2}$ and $K_{2n+3}$ that are different in the sense that they are not shellable. As pointed out by one of the referees of the preliminary version of this paper, Jan Kyncl asked on {\tt math}{\it overflow} \cite{Ky13} about more general drawings with $H(n)$ crossings. In a later response to that Overflow post, it was noted that such drawings may be obtained from antipodal drawings on the sphere, very similar to the construction in this paper. Still, several questions from \cite{Ky13} remained unanswered, and we answer some of them in Section~\ref{sect:final comments}.

In this paper we give a very general construction of Hill drawings (that includes in particular both mentioned examples). The proof of extremality of these drawings may be a candidate for a book proof (the count of the crossings fits into a couple of lines, as opposed to two pages in \cite{SchBook18}). Our proof is extremely simple and effective, and describes a family of very symmetric, yet very diverse drawings.

Let $\SS^2$  be the unit sphere in $\RR^3$. For any two points $p,q\in \SS^2$ consider the great circle through $p$ and $q$ (the great circle is unique unless $q$ is antipodal to $p$ in which case there are many). The shorter of the two segments on this circle from $p$ to $q$ is called a \emph{geodesic arc} (or just a \emph{geodesic}). Any geodesic arc joining two antipodal points in $\SS^2$ is a half of a great circle and will be referred to as a \emph{half-circle}.

A \emph{geodesic drawing} of a graph $G$ on $\SS^2$ is a drawing in which all edges are drawn as geodesic arcs. We define the \emph{geodesic crossing number} of the graph $G$ on the sphere as the minimum number of crossings of edges of $G$ in a geodesic drawing of the graph, and denote it by $\crg(G)$.

A set $P$ of points in $\SS^2$ is \emph{in general position} if no three points in $P$ lie on a common great circle in the sphere. The main result of this paper is the following.

\begin{theorem}\label{thm:main}
Let $k\ge3$ be a positive integer and let $P$ be a set of $k$ points in general position in $\SS^2$. Let $\mathcal S$ be obtained from $P$ by adding, for each $p\in P$, its antipodal point $\bar p$ into $\mathcal S$. If each pair $\{p,\bar p\}$ in $\mathcal S$ is joined by a geodesic in such a way that no two of these geodesics cross each other, then the geodesic drawing $\widehat D_n$ obtained from $n=2k$ points in $\mathcal S$ by joining any two points not in the same pair by a geodesic arc, together with the $k$ geodesics joining antipodal pairs, is a geodesic drawing of the complete graph $K_n$ with $H(n)$ crossings. This drawing has the following additional properties. By deleting any point from any pair in $\mathcal S$, we obtain a drawing of $K_{n-1}$ with precisely $H(n-1)$ crossings, and by adding any new point (in general position with respect to $\mathcal S$) and adding geodesics from that point to all other points, we obtain a geodesic drawing of $K_{n+1}$ with $H(n+1)$ crossings.
\end{theorem}

The theorem implies that $\crg(K_n)\le H(n)$ for every positive integer $n$.
This result is surprising in two ways. Firstly, it is known that the rectilinear crossing number (geodesic version in the Euclidean plane) of complete graphs is strictly larger than the usual crossing number (see \cite{LoVeWaWe}).
So, assuming the Hill conjecture, it is surprising that the geodesic crossing number in the sphere is not different.
Secondly, the abundance of obtained Hill drawings is also quite unexpected.

In 1968, Moon \cite{Moon65} proved that a random set of $n$ points on the unit sphere $\SS^2$ in $\RR^3$ joined by geodesics (segments of great circles) gives rise to a drawing whose number of crossings asymptotically approaches the Hill value $H(n)$. Our construction gives a very simple explanation of this phenomenon. Indeed, in a forthcoming work \cite{MoWe} the following result with several interesting consequences is derived.

\begin{theorem}[Mohar and Wesolek \cite{MoWe}]
Let $\mu$ be an antipodally symmetric probability distribution on the unit sphere $\SS^2$ such that for every great circle $Q\subset \SS^2$, $\mu(Q)=0$. Then a $\mu$-random set of $n$ points on $\SS^2$ joined by geodesics (segments of great circles) gives rise to a drawing $D_n$ of the complete graph $K_n$ such that $cr(D_n)/H(n)=1$ a.a.s.
\end{theorem}

\section{On the geodesic crossing number of $K_n-I_t$}

For integers $n\ge5$ and $t$, $0\le t\le \frac{n}{2}$, let $M_{n,t}$ denote the graph obtained from $K_n$ by removing a matching of size $t$. Specifically, we will write $M_n = M_{n,\lfloor n/2\rfloor}$ be the complete graph $K_n$ minus a perfect matching (when $n$ is even) or a near-perfect matching (when $n$ is odd). Note that $M_n$ is isomorphic to the complete multipartite graph $K_{2,2,\dots,2}$ with $k = \tfrac{n}{2}$ parts of size 2, whenever $n$ is even. The edge-set of this graph consists of $\binom{k}{2}$ 4-cycles, each of which joins two parts of size 2 and is called a \emph{basic $4$-cycle}.

\begin{theorem}
\label{thm:drawing D_n}
Let $k\ge3$ be an integer and let $P$ be a set of $k$ points in general position on the unit sphere. Let $\hat P$ be the set of $n=2k$ points obtained from $P$ by adding for each $p\in P$ its antipodal point $\bar p$. For any two points in $\hat P$ that are not antipodal to each other, draw the geodesic segment of the great circle joining them. Then the resulting drawing $D_n=D_n(P)$ of the graph $M_n$ has precisely $\tfrac{1}{4} k(k-1)(k-2)(k-3)$ crossings.
By adding any geodesic half-circle between a pair of antipodal points $p,\bar p$ $(p\in P)$, we obtain $\tfrac{1}{2}(k-1)(k-2)$ additional crossings.
\end{theorem}

\begin{proof}
Let $D_n$ be the geodesic drawing described in the statement. Note that every pair of points $p,q\in P$ together with their antipodes $\bar p, \bar q$ determines a great circle $Q_{pq}$ that consists of four edges forming the basic 4-cycle between $\{p,\bar p\}$ and $\{q,\bar q\}$. Any two such great circles $Q_{pq}$ and $Q_{rs}$ cross twice and make two crossings if $\{p,q\}\cap \{r,s\}= \emptyset$. If $|\{p,q\}\cap \{r,s\}| = 1$, then they do not cross. Thus, the edges in each $Q_{pq}$ participate in precisely $2\binom{k-2}{2} = (k-2)(k-3)$ crossings. By summing up these numbers over all $\binom{k}{2}$ possibilities for the pair $\{p,q\}$, we count each crossing twice, so
$$
   cr(D_n) = \frac{1}{2}\binom{k}{2} (k-2)(k-3) = \tfrac{1}{4} k(k-1)(k-2)(k-3).
$$
By adding any great circle through two antipodal points in $\hat P$, we separate $k-1$ of the points in $\hat P$ from their antipodal pairs. There are precisely $(k-1)(k-2)$ edges joining them. Because of the antipodal symmetry of the drawing $D_n$, precisely half of these edges cross each half-circle. Thus, each half-circle is crossed $\tfrac{1}{2}(k-1)(k-2)$ many times.
\end{proof}

We say that the set $P$ of points in $\SS^2$ has \emph{strength} $s$ if there is a choice of half-circles joining each point in $P$ with its antipodal point $\bar p$ such that these half-circles cross each other $s$ times.

\begin{corollary}
If a set $P'\subseteq P$ has strength $s$ and $t=|P\setminus P'|$, then the drawing $D_n(P)$ can be extended to a drawing of the graph $M_{n,t}$ with $H(n) - \tfrac{1}{2}t(k-1)(k-2) + s$ crossings.
\end{corollary}

\begin{proof}
We extend the drawing $D_n$ by adding half-circles joining the antipodal pairs $p,\bar p$ for $p\in P'$ so that these half-circles make $s$ crossings among each other. By Theorem \ref{thm:drawing D_n}, the number of crossings is $\tfrac{1}{4} k(k-1)(k-2)(k-3) + \tfrac{1}{2}(k-1)(k-2)|P'| + s$, which is the same as the number in the corollary since $\tfrac{1}{4} k(k-1)(k-2)(k-3) + \tfrac{1}{2}k(k-1)(k-2) = H(n)$.
\end{proof}

It is easy to see that there are many sets of strength 0. This leads us to the following conjecture.

\begin{conjecture}
\label{conj:cr M_nt}
The crossing number of $M_{n,t}$ is equal to $$H(n) - \tfrac{1}{2}t(k-1)(k-2).$$
\end{conjecture}

The following corollary establishes Theorem \ref{thm:main}.

\begin{corollary}
Let $P\subset \SS^2$ and the drawing $D_n$ be as in Theorem \ref{thm:drawing D_n}.
If $P$ has strength $0$, then $D_n$ can be extended to a geodesic drawing $\widehat D_n$ of the complete graph $K_n$ with $H(n)$ crossings. This drawing has the following additional properties:

(a) The drawing is antipodally symmetric.

(b) For every vertex $v$ of $K_n$, the edges incident with $v$ participate in precisely $\tfrac{1}{16} (n-2)^2(n-4)$ crossings.

(c) By deleting any point from $\hat P$, we obtain a drawing of $K_{n-1}$ with precisely $H(n-1)$ crossings.

(d) By adding any new point (in general position with respect to $P$) and adding geodesics from that point to $\hat P$, we obtain a geodesic drawing of $K_{n+1}$ with $H(n+1)$ crossings.
\end{corollary}

\begin{proof}
Statements (a)--(c) are easy observations and their proof is left for the reader. To prove (d), let $Q = P \cup \{q\}$, where $q$ is the added point. Consider the corresponding drawing of $K_{n+2}$ for $\hat Q$. Note that $Q$ may no longer have strength 0, but since $P$ has strength 0, there is a drawing where the only half-circle intersecting other half-circles is the half-circle joining $q$ and $\bar q$. All these added crossings disappear after removing $\bar q$, and thus by (b), the extended drawing of $K_{n+1}$ has $H(n+2) - \tfrac{1}{16} n^2(n-2) = H(n+1)$ crossings.
\end{proof}

\section{Final comments}
\label{sect:final comments}

The referees of the original submission of this manuscript pointed out that the idea of using antipodal drawings on the sphere appeared earlier on Math Overflow related to a question of Jan Kyncl \cite{Ky13}, who originally asked whether any Hill drawings different from the cylindrical and the 2-page drawings from \cite{HaHi62,BlKo63} are known. In a later edit to the original question, the antipodal drawings of $K_{2n}$ were mentioned and it was asked whether all such drawings are cylindrical. Here we give some answers.

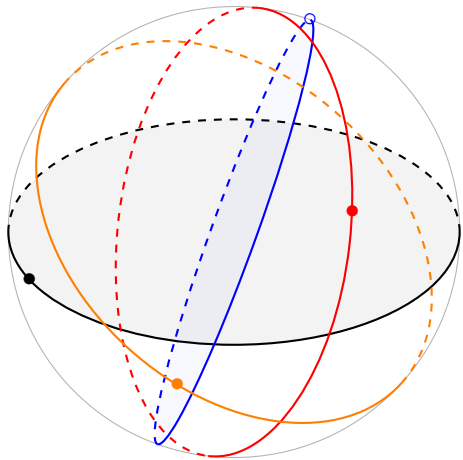
\begin{figure}
  \centering
  \def\r{3}
  \tdplotsetmaincoords{60}{125}
  \begin{tikzpicture}[tdplot_main_coords]
    \draw[tdplot_screen_coords,thin,black!30] (0,0,0) circle (\r);
    \tdplotCsDrawLatCircle%
      [thick,tdplotCsFill/.style={opacity=0.05}]{\r}{0}
    \tdplotCsDrawGreatCircle%
      [blue,thick,tdplotCsFill/.style={opacity=0.03}]{\r}{120}{110}
    \tdplotCsDrawGreatCircle%
      [red,thick]{\r}{-30}{70}
    \tdplotCsDrawGreatCircle%
      [orange,thick]{\r}{90}{36}
    \tdplotCsDrawPoint[blue]{\r}{180}{36}
    \tdplotCsDrawPoint{\r}{-30}{90}
    \tdplotCsDrawPoint[red]{\r}{72.3}{60}
    \tdplotCsDrawPoint[orange]{\r}{20}{104}
  \end{tikzpicture}
  \caption{Four points on $\SS^2$ and their antipodes determine disjoint half-circles. The points and their great circles are shown (in different colors), the corresponding half-circles are taken in the clockwise direction starting from each of the points.}\label{fig:4antipodal}
\end{figure}

\begin{lemma}\label{lem:blowup halfcircle}
  For every $\varepsilon>0$, every integer $n\ge1$ and every half-circle $C$ on the unit sphere $\SS^2$, the $\varepsilon$-neighborhood of $C$ contains $n$ pairwise disjoint half-circles.
\end{lemma}

\begin{proof}
  Suppose that $p$ and $\bar p$ are the endpoints of $C$. Then it is easy to see that there is a half-circle $C_1$ that is in the $\frac{\varepsilon}{n}$-neighborhood of $C$ ``below" $C$ and that crosses the half-circle containing $C$ in a point of $\overline{C}$ close to $p$. For the next half-circle $C_2$ we repeat the same below $C_1$. Repeating, we obtain $n$ pairwise disjoint half-circles as requested.
\end{proof}

By using Lemma \ref{lem:blowup halfcircle} repeatedly on any initial set of pairwise disjoint half-circles, we obtain diverse sets of antipodal points of strength 0. 

\begin{corollary}\label{cor:construction}
Starting with any antipodal Hill drawing $D$ and replacing each half-circle in $D$ by a collection of pairwise disjoint half-circles in a small neighborhood, one obtains Hill drawings of larger complete graphs.
\end{corollary}

If we start with $D$ consisting of a single half-circle and use the lemma just once, we obtain a cylindrical drawing. If we start with two half-circles, say half of the equator and another one joining the south and the north pole, we obtain a four-partite version of cylindrical drawings. We can start with more than two half-circles of strength 0. Figure \ref{fig:4antipodal} shows a possibility of four half-circles to start with. Then we obtain yet another kinds of Hill drawings (six-partite or eight-partite). Let us note that Camacho et al.\ \cite{AMUC19} proved that there is no tripartite version of cylindrical drawings yielding the Hill bound.

Note that one can interchangeably take the ``below" and the ``above" versions of the construction described in the proof of Lemma \ref{lem:blowup halfcircle}. Moreover, one can use the larger neighborhood of $C\cup \overline{C}$ as long as it is disjoint from other half-circles. This yields a great variety of Hill drawings that answer the question by Kyncl in \cite{Ky13}. It also gives further examples of ``highly" non-shellable Hill drawings and greatly extends such examples found in \cite{AAFRV14}. 

\bibliographystyle{abbrv}
\bibliography{Antipodal}

\end{document}